\documentclass[11pt]{article}
\usepackage[margin=1in]{geometry} 
\usepackage{amssymb,amsfonts,amsmath,bbm,mathrsfs,stmaryrd,mathtools}
\usepackage{xcolor}
\usepackage{url}

\usepackage{extarrows}

\usepackage[shortlabels]{enumitem}
\usepackage{tensor}

\usepackage[T1]{fontenc}

\usepackage[colorlinks,
linkcolor=black!75!red,
citecolor=blue,
pdftitle={},
pdfproducer={pdfLaTeX},
pdfpagemode=None,
bookmarksopen=true,
bookmarksnumbered=true,
backref=page]{hyperref}

\usepackage{tikz}
\usetikzlibrary{arrows,calc,decorations.pathreplacing,decorations.markings,intersections,shapes.geometric,through,fit,shapes.symbols,positioning,decorations.pathmorphing}

\usepackage{braket}

\usepackage[amsmath,thmmarks,hyperref]{ntheorem}
\usepackage{cleveref}

\newcommand\nthalias[1]{\AddToHook{env/#1/begin}{\crefalias{lemma}{#1}}}

\nthalias{definition}
\nthalias{example}
\nthalias{examples}
\nthalias{remark}
\nthalias{remarks}
\nthalias{convention}
\nthalias{notation}
\nthalias{construction}
\nthalias{sketch}
\nthalias{theoremN}
\nthalias{propositionN}
\nthalias{corollaryN}
\nthalias{lemma}
\nthalias{proposition}
\nthalias{corollary}
\nthalias{theorem}
\nthalias{conjecture}
\nthalias{question}
\nthalias{assumption}

\creflabelformat{enumi}{#2#1#3}

\crefname{section}{Section}{Sections}
\crefformat{section}{#2Section~#1#3} 
\Crefformat{section}{#2Section~#1#3} 

\crefname{subsection}{\S}{\S\S}
\AtBeginDocument{%
  \crefformat{subsection}{#2\S#1#3}%
  \Crefformat{subsection}{#2\S#1#3}%
}

\crefname{subsubsection}{\S}{\S\S}
\AtBeginDocument{%
  \crefformat{subsubsection}{#2\S#1#3}%
  \Crefformat{subsubsection}{#2\S#1#3}%
}

\theoremstyle{plain}

\newtheorem{lemma}{Lemma}[section]
\newtheorem{proposition}[lemma]{Proposition}
\newtheorem{corollary}[lemma]{Corollary}
\newtheorem{theorem}[lemma]{Theorem}

\theoremstyle{plain}
\theoremnumbering{Alph}
\newtheorem{theoremN}{Theorem}

\theoremstyle{plain}
\theorembodyfont{\upshape}
\theoremsymbol{\ensuremath{\blacklozenge}}

\newtheorem{definition}[lemma]{Definition}
\newtheorem{example}[lemma]{Example}

\newtheorem{remark}[lemma]{Remark}
\newtheorem{remarks}[lemma]{Remarks}

\crefname{definition}{definition}{definitions}
\crefformat{definition}{#2definition~#1#3} 
\Crefformat{definition}{#2Definition~#1#3} 

\crefname{example}{example}{examples}
\crefformat{example}{#2example~#1#3} 
\Crefformat{example}{#2Example~#1#3} 

\crefname{examples}{example}{examples}
\crefformat{examples}{#2example~#1#3} 
\Crefformat{examples}{#2Example~#1#3} 

\crefname{remark}{remark}{remarks}
\crefformat{remark}{#2remark~#1#3} 
\Crefformat{remark}{#2Remark~#1#3} 

\crefname{remarks}{remark}{remarks}
\crefformat{remarks}{#2remark~#1#3} 
\Crefformat{remarks}{#2Remark~#1#3} 

\crefname{convention}{convention}{conventions}
\crefformat{convention}{#2convention~#1#3} 
\Crefformat{convention}{#2Convention~#1#3} 

\crefname{notation}{notation}{notations}
\crefformat{notation}{#2notation~#1#3} 
\Crefformat{notation}{#2Notation~#1#3} 

\crefname{table}{table}{tables}
\crefformat{table}{#2table~#1#3} 
\Crefformat{table}{#2Table~#1#3}

\crefname{lemma}{lemma}{lemmas}
\crefformat{lemma}{#2lemma~#1#3} 
\Crefformat{lemma}{#2Lemma~#1#3} 

\crefname{proposition}{proposition}{propositions}
\crefformat{proposition}{#2proposition~#1#3} 
\Crefformat{proposition}{#2Proposition~#1#3} 

\crefname{corollary}{corollary}{corollaries}
\crefformat{corollary}{#2corollary~#1#3} 
\Crefformat{corollary}{#2Corollary~#1#3} 

\crefname{theorem}{theorem}{theorems}
\crefformat{theorem}{#2theorem~#1#3} 
\Crefformat{theorem}{#2Theorem~#1#3} 

\crefname{theoremN}{theorem}{theorems}
\crefformat{theoremN}{#2theorem~#1#3} 
\Crefformat{theoremN}{#2Theorem~#1#3} 

\crefname{enumi}{}{}
\crefformat{enumi}{#2#1#3}
\Crefformat{enumi}{#2#1#3}

\crefname{assumption}{assumption}{Assumptions}
\crefformat{assumption}{#2assumption~#1#3} 
\Crefformat{assumption}{#2Assumption~#1#3} 

\crefname{construction}{construction}{Constructions}
\crefformat{construction}{#2construction~#1#3} 
\Crefformat{construction}{#2Construction~#1#3} 

\crefname{equation}{}{}
\crefformat{equation}{(#2#1#3)} 
\Crefformat{equation}{(#2#1#3)}

\numberwithin{equation}{section}

\theoremstyle{nonumberplain}
\theoremsymbol{\ensuremath{\blacksquare}}

\newtheorem{proof}{Proof}
\newcommand\pf[1]{\newtheorem{#1}{Proof of \Cref{#1}}}
\newcommand\pff[3]{\newtheorem{#1}{Proof of \Cref{#2}#3}}

\newcommand\bC{{\mathbb C}}

\newcommand\bG{{\mathbb G}}
\newcommand\bH{{\mathbb H}}

\newcommand\bK{{\mathbb K}}
\newcommand\bL{{\mathbb L}}

\newcommand\bS{{\mathbb S}}

\newcommand\bU{{\mathbb U}}

\newcommand\bZ{{\mathbb Z}}

\newcommand\cC{{\mathcal C}}

\newcommand\fp{{\mathfrak p}}

\DeclareMathOperator{\id}{id}

\DeclareMathOperator{\im}{\mathrm{im}}
\DeclareMathOperator{\diag}{diag}

\DeclareMathOperator{\Irr}{Irr}
\DeclareMathOperator{\spn}{span}
\DeclareMathOperator{\Spec}{Spec}
\DeclareMathOperator{\cCL}{\mathcal{CL}}

\newcommand{\cat}[1]{\textsc{#1}}

\newcommand{\qedhere}{\mbox{}\hfill\ensuremath{\blacksquare}}

\newcommand{\xrightarrowdbl}[2][]{%
  \xrightarrow[#1]{#2}\mathrel{\mkern-14mu}\rightarrow
}

\title{Orbit misbehavior, isotropy discontinuity, and large isotypic components}
\author{Alexandru Chirvasitu}

\begin{document}

\date{}

\newcommand{\Addresses}{{
  \bigskip
  \footnotesize

  \textsc{Department of Mathematics, University at Buffalo}
  \par\nopagebreak
  \textsc{Buffalo, NY 14260-2900, USA}  
  \par\nopagebreak
  \textit{E-mail address}: \texttt{achirvas@buffalo.edu}


}}

\maketitle

\begin{abstract}
  Let $\mathbb{G}$ be a compact Hausdorff group acting on a compact Hausdorff space $X$, $\alpha$ an irreducible $\mathbb{G}$-representation, and $C(X)$ the $C^*$-algebra of complex-valued continuous functions on $X$. We prove that the isotypic component $C(X)_{\alpha}$ is finitely generated as a module over the invariant subalgebra $C(X/\mathbb{G})\subseteq C(X)$ precisely when the map sending $x\in X$ to the dimension of the space of vectors in $\alpha$ invariant under the isotropy group $\mathbb{G}_x$ is locally constant. This (a) specializes back to an observation of De Commer-Yamashita equating the finite generation of all $C(X)_{\alpha}$ with the Vietoris continuity of $x\mapsto \mathbb{G}_x$, and (b) recovers and extends Watatani's examples of infinite-index expectations resulting from non-free finite-group actions.

  We also show that the action of a compact group $\mathbb{G}$ on the maximal equivariant compactification on the disjoint union of its Lie-group quotients has tubes about all orbits precisely when $\mathbb{G}$ is Lie. This is the converse (via a canonical construction) of the well-known fact that actions of compact Lie groups on Tychonoff spaces admit tubes.
\end{abstract}


\vspace{.5cm}



\section*{Introduction}

For every {\it representation} \cite[\S 2]{rob} of a compact (always Hausdorff in the sequel) group $\bG$ on a Banach space $F$ and every irreducible (hence finite-dimensional \cite[Corollary 5.8]{rob}) representation $\alpha\in\Irr(\bG)$ we have the corresponding {\it spectral projection} (the $P_{e_{\alpha}}$ of \cite[Definition 4.21]{hm5})
\begin{equation*}
  F
  \xrightarrowdbl{\quad\pi_{\alpha}\quad}
  F_{\alpha}
  :=
  \left\{v\in F\ |\ \spn\bG v\cong \alpha^{n}\text{ for some }n\in \bZ_{>0}\right\}
  \le F
  ,\quad
  \alpha\in\Irr(\bG)
\end{equation*}
onto the {\it ($\alpha$-)isotypic component}. This all applies in particular to the $\bG$-action $s\triangleright f:=f(s^{-1}\triangleright\bullet)$ on the $C^*$-algebra $F:=C(X)$ of complex-valued continuous functions on $X$ for a compact Hausdorff left $\bG$-space $\bG\times X\xrightarrow{\triangleright}X$. 

This paper was initially motivated by the cursory \cite[Remark 2 preceding Proposition 3.4]{dcy}, to the effect that the $C(X)_{\alpha}$ are {\it all} finitely generated (f.g. for short) over the invariant subalgebra $C(X)_{\cat{triv}} = C(X)^{\bG}\cong C(X/\bG)$ of $C(X)$ precisely when
\begin{equation}\label{eq:x2gx}
  X\ni x
  \xmapsto{\quad}
  \bG_x:=\text{isotropy subgroup }
  \left\{s\in \bG\ |\ s\triangleright x = x\right\}
  \le \bG
\end{equation}
is continuous for the {\it Vietoris topology} \cite[\S III.4.2]{johnst_stone} on the space of closed subsets (or subgroups) of $\bG$. This means that, regarded as a map to $2^{\bG}$ (the power-set of $\bG$), \Cref{eq:x2gx} is both 
\begin{itemize}
\item {\it upper semicontinuous} (sometimes {\it hemicontinuous}: \cite[Definition 17.2]{ab_inf-dim-an}) in the sense that
  \begin{equation*}
    \forall x\in X
    ,\ 
    \forall\text{ neighborhood }U\supseteq \bG_x
    ,\quad
    \{x'\in X\ |\ \bG_{x'}\subseteq U\}\subseteq X
    \text{ is a neighborhood of }x
  \end{equation*}
  (this much is automatic for every action \cite[Proposition 10.6]{hm5});

\item and {\it lower semicontinuous}:
  \begin{equation*}
    \forall x\in X
    ,\
    \forall\text{ open }U\text{ with }\bG_x \cap U\ne \emptyset
    ,\quad
    \{x'\in X\ |\ \bG_{x'}\cap U\ne \emptyset\}\subseteq X
    \text{ is a neighborhood of }x.
  \end{equation*}
\end{itemize}
That motivating observation can be retrieved (\Cref{cor:iffviet}) as a consequence of a ``robust'' version thereof (\Cref{th:tameiffcontrk}), characterizing the $\alpha$ for which finite generation holds in terms of isotropy behavior:

\begin{theoremN}\label{thn:fg}
  Let $\bG$ be a compact group acting on a compact Hausdorff space $X$. The isotypic component $C(X)_{\alpha}$, $\alpha\in\Irr(\bG)$ is finitely generated as a $C(X)^{\bG}$-module if and only if the function
  \begin{equation*}
    X\ni x
    \xmapsto{\quad}
    \dim\alpha^{\bG_x}
    \in \bZ_{\ge 0}
    ,\quad
    \alpha^{\bH}:=\text{$\bH$-invariant vectors in the carrier space of $\alpha$}
  \end{equation*}
  is locally constant.  \qedhere
\end{theoremN}

Having thus accorded some attention to irregularities of the isotropy-group map \Cref{eq:x2gx}, \Cref{se:tubext} switches focus to the adjacent issue of action misbehavior for non-Lie $\bG$. By way of a brief reminder, recall that an {\it open} or {\it closed tube} \cite[\S II.4]{bred_cpct-transf} about an orbit $\bG x$ is a $\bG$-equivariant homeomorphism of 
\begin{equation*}
  \bG\times_{\bG_x}Y
  :=
  \bG\times Y/\left((ts,y) = (t,sy),\quad\forall t\in \bG,\ s\in \bG_x,\ y\in Y\right)    
\end{equation*}
onto an open or, respectively, closed neighborhood of $\bG x$ for a $\bG_x$-space $Y$ (then referred to as a {\it slice} \cite[Definition 4.1]{bred_cpct-transf}). Unqualified tubes are typically understood to be open. 


It is well known (\cite[Theorem II.5.4]{bred_cpct-transf}, \cite[Theorem I.5.7]{td_transf-gp}, etc.) that all orbits of compact-Lie-group actions on {\it Tychonoff (or $T_{3\frac 12}$)} \cite[Definition 14.8]{wil_top} spaces admit tubes, and (the proof of) \Cref{th:nonlie-nonop} inverts that implication canonically:

\begin{theoremN}\label{thn:lie}
  A compact group $\bG$ is Lie if and only if its action on the {\it universal equivariant compactification} \cite[\S 2.8]{dvr_puc_1977} of
  \begin{equation*}
    \coprod_{\bK}\left(\text{Lie-group quotients }\bG/\bK\right)
  \end{equation*}
  admits tubes around all of its orbits.  \qedhere
\end{theoremN}

An earlier, more involved and less direct attempt at proving \Cref{thn:fg} resorted to an analysis of an arbitrary compact-group action based on restricting attention to tubes. It quickly became apparent that the Lie/non-Lie dichotomy would have to be central to any such attempt: tubes are significantly better behaved when $\bG$ is Lie, which in turn forces any such route to \Cref{thn:fg} into a reduction of the general problem to Lie-group approximations of arbitrary compact $\bG$. \Cref{se:tubext} (summarized here in the form of \Cref{thn:lie}) is an offshoot of those earlier considerations.

\subsection*{Acknowledgements}

This work is partially supported by NSF grant DMS-2001128, and is part of the project Graph Algebras partially supported by EU grant HORIZON-MSCA-SE-2021 Project 101086394.

Early drafts have benefited from helpful pointers and remarks from K. De Commer, P. M. Hajac, G. Luk\'acs, B. Passer, P. So{\l}tan and M. Yamashita. I am also grateful for the anonymous referee's instructive comments, questions and suggestions. 

Although the aforementioned \cite[Remark 2, p.725]{dcy} does not cite a source for the claim recovered below as \Cref{cor:iffviet}, in the late stages of preparing a draft of the present paper it was brought to my attention that there is work \cite{bdc_pre} (albeit not in circulation at the time of this writing) proving that result and presumably overlapping the material below in a few ways. I am very grateful to K. De Commer for informing me of this. 


\section{Orbits, isotropy jumps, and (failure of) finite generation}\label{se:orb-jump}

Topological (mostly compact) groups are assumed Hausdorff, and $\bG$-spaces $X$ are at least {\it Tychonoff (or $T_{3\frac 12}$)} \cite[Definition 14.8]{wil_top}, and usually themselves compact Hausdorff. A few basic observations will be implicit in much that follows:
\begin{itemize}
\item For a compact Hausdorff $\bG$-space with compact $\bG$ the restrictions 
  \begin{equation}\label{eq:tg}
    C(X)_{\alpha}
    \xrightarrowdbl{\quad\text{restriction}\quad}
    C(Z)_{\alpha}
    ,\quad
    \text{closed $\bG$-invariant }Z\subseteq X
  \end{equation}
  are onto by the {\it Tietze-Gleason} \cite[Theorem I.2.3]{bred_cpct-transf}.

\item Whence the non-zero $C(X)_{\alpha}$ are precisely those with $\alpha\le C(\bG/\bG_x)$ for some isotropy group $\bG_x$. Equivalently \cite[pp.82-83]{rob}, $\alpha$ must have non-zero vectors invariant under some $\bG_x$.   
\end{itemize}

The following example is a paraphrase (and extension) of sorts of \cite[p.711, 2$^{nd}$ paragraph]{dcy}, where $\bG:=\bS^1$. See also \cite[Proposition 2.8.2]{wat_indexcast} for a similar discussion concerning non-free finite-group actions with ``most'' isotropy groups trivial. 

\begin{example}\label{ex:goncone}
  Extend the (left, say) translation self-action of a compact group $\bG$ in the obvious fashion to the {\it cone} \cite[Chapter 0, pp.8-9]{hatch_at}
  \begin{equation*}
    X:=\cC \bG:=\left(\bG\times I\right)/\left(\bG\times\{0\}\right)
    ,\quad
    I:=[0,1]. 
  \end{equation*}
  I claim that the only isotypic component $C(X)_{\alpha}$ that is finitely generated as a module over $C(X)_{\cat{triv}}$ is $C(X)_{\cat{triv}}$ itself. 

  Note first that
  \begin{equation*}
    C(X)_{\cat{triv}} = C(X)^{\bG} \cong C(X/\bG) \cong C(I), 
  \end{equation*}
  consisting of continuous functions defined arbitrarily on any single interval
  \begin{equation*}
    I\cong \{s\}\times I\subset X
    ,\quad
    \text{fixed }s\in \bG
  \end{equation*}
  and then extended uniquely to $X$ by equivariance. Now fix $\cat{triv}\ne \alpha\in \Irr(\bG)$. Now observe:
  \begin{itemize}
  \item On the one hand all elements of $C(X)_{\alpha}$ must vanish at the tip $p_0$ of the cone $X=\cC\bG$: regarding an $f\in C(X)_{\alpha}$ as a function on $\bG\times I$, its restriction to the slice $\bG\times \{0\}$ collapsed in the cone is both constant and an element of $C(\bG)_{\alpha}$ (so zero).
    
  \item On the other hand, every $0\le f\in C(X)^{\bG}$ vanishing at $p_0$ can be recovered as
    \begin{equation*}
      f=\pi(g^* g),\ g\in C(X)_{\alpha}
      \quad
      \text{for the usual expectation }
      C(X)\xrightarrow{\quad \pi:=\pi_{\cat{triv}}\quad}C(X)^{\bG}.
    \end{equation*}
  \end{itemize}
  It follows that the ideal
  \begin{equation}\label{eq:spanideal}
    \spn\left\{\pi(g^* g')\ |\ g,g'\in C(X)_{\alpha}\right\}\trianglelefteq C(X)^{\bG}\cong C(I)
  \end{equation}
  is precisely the ideal
  \begin{equation}\label{eq:m0ideal}
    M_{0}:=\left\{f\in C(I)\ |\ f(0)=0\right\}\trianglelefteq C(I)
    \text{ of \cite[Theorem 4.6]{gj_rings}}. 
  \end{equation}
  If $C(X)_{\alpha}$ were finitely generated over $C(I)$ so too would \Cref{eq:m0ideal}; on the other hand, \cite[Problem 14C]{gj_rings} shows that (the image of) $M_0$ is not finitely generated even after passing to the quotient ring $C(I)/O_0$ by the ideal
  \begin{equation*}
    O_0:=\left\{f\in C(I)\text{ vanishing in a neighborhood of }0\right\}\trianglelefteq C(I)
  \end{equation*}
  of \cite[Problem 4I]{gj_rings}. 
\end{example}

\begin{remarks}\label{res:topfg}
  \begin{enumerate}[(1),wide=0pt]
  \item\label{item:res:topfg:is.tfg} It is also clear from \Cref{ex:goncone} that isotypic components can easily be {\it topologically} finitely generated over $C(X)^{\bG}$ without being algebraically so, in the sense of containing a dense finitely-generated module. Indeed, all of them are (topologically f.g.): they are easily seen to be of the form
    \begin{equation*}
      C(\cC\bG)_{\alpha}
      \cong
      C_0((0,1])\otimes C(\bG)_{\alpha}
      \quad
      \left(\text{ordinary algebraic tensor product}\right)
    \end{equation*}
    (with $C_0$ denoting, as usual \cite[General notation]{wo}, continuous functions on a locally compact space vanishing at infinity), and hence \cite[pre Theorem 1.1]{gog_top-fg} each is respectively topologically generated by
    \begin{equation*}
      \dim_{\bC}C(\bG)_{\alpha}
      \xlongequal[]{\quad \text{\cite[pre Corollary 5.7]{rob}}\quad}
      \left(\dim\alpha\right)^2
    \end{equation*}
    elements.

  \item On the other hand, it is also perfectly possible for isotypic components $C(X)_{\alpha}$ not to be even topologically f.g. over $C(X/\bG)$: \Cref{ex:omegaomega}.    
  \end{enumerate}  
\end{remarks}

\begin{example}\label{ex:omegaomega}
  Construct the action $\bG\times X\to X$, for instance, as follows:
    \begin{itemize}[wide=0pt]
    \item Set $X_0:=\bG\times Y$ with the left-hand-factor translation $\bG$-action for any locally compact (Hausdorff) that is not {\it $\sigma$-compact} (i.e. \cite[\S I.3, p.19]{ss_countertop} a union of countably many compact subspaces).

      Examples include any uncountable discrete set or 
      \begin{equation*}
        \begin{aligned}
          Y:=[0,\Omega)
          &= \left\{\text{set of countable ordinals}\right\}\\
          &= \left\{\text{set of ordinals}<\Omega\right\}\\
          &= \Omega
            :=\text{first uncountable ordinal}
        \end{aligned}        
      \end{equation*}
      endowed with the {\it order topology} of \cite[\S II.39]{ss_countertop} (the space also denoted by $[0,\Omega)$ in \cite[\S II.42]{ss_countertop}).      

    \item Extend that action in the only way possible to the {\it one-point compactification} \cite[\S II.35]{ss_countertop} $X:=X_0\sqcup\{*\}$ of $X_0$.      
    \end{itemize}
    The quotient space $X/\bG$ is the one-point compactification $\overline{Y}=Y\sqcup\{*\}$ of $Y$, and {\it no} $C(X)_{\alpha}$ is topologically finitely generated for non-trivial $\alpha$. One can argue as in \Cref{ex:goncone}: note that every $\pi(g^* g')\in C(X/\bG)\cong C(\overline{Y})$ with $g,g'\in C(X)_{\alpha\ne \cat{triv}}$ vanishes along some non-singleton closed subset $F\supset *\in \overline{Y}$ because the assumed non-$\sigma$-compactness ensures that $\{*\}$ is not a {\it $G_{\delta}$ set} (i.e. \cite[\S I.1, p.4]{ss_countertop} a countable intersection of open sets). Indeed: $\{*\}$'s being $G_{\delta}$ is by the very definition of the one-point-compactification topology equivalent to its complement $Y=\overline{Y}\setminus \{*\}$ a countable union of closed (hence compact) subsets of $Y$. The singleton $\{*\}$ being non-$G_\delta$, finitely many such $\pi(g^* g')$ cannot generate the entire ideal defined on the left-hand side of \Cref{eq:spanideal}. 
\end{example}

The phenomenon underlying \Cref{ex:goncone} (and \cite[Proposition 2.8.2 and Example 2.8.4]{wat_indexcast}) is of ampler scope, and the pertinent proof techniques allow the generalization of \Cref{cor:watgen}, in turn following from \Cref{th:tameiffcontrk}. The latter is very much in line with \cite[Remark post Theorem 3.3]{dcy}, noting the ill behavior resulting, for an action of a compact group $\bG$ on a compact space $X$, from the Vietoris-discontinuity of \Cref{eq:x2gx}. 

It will be convenient to have short-hand vocabulary for referring to finite generation of isotypic components.

\begin{definition}\label{def:tame}
  Let $\bG$ be a compact group acting on a compact Hausdorff space $X$. A representation $\alpha\in\Irr(\bG)$ is {\it tame (for the given action)} if $C(X)_{\alpha}$ is finitely $C(X)^{\bG}$-generated.

  Variations of the phrase, such as tameness {\it over $X$}, also appear and should be self-explanatory when they do. We also refer to $X$ or the action as {\it $\alpha$-taming}. 
\end{definition}

For $\bG$-representations $\alpha$ and closed subgroups $\bH\le \bG$ we write 
\begin{equation*}
  \alpha^{\bH}:=\{\text{$\bH$-invariant vectors in $\alpha$}\}.
\end{equation*}

\begin{theorem}\label{th:tameiffcontrk}
  A representation $\alpha\in \Irr(\bG)$ is tame for an action of a compact group $\bG$ on compact Hausdorff $X$ if and only if
  \begin{equation}\label{eq:agxcont}
    X\ni x
    \xmapsto{\quad}
    \dim\alpha^{\bG_x}
    \in \bZ_{\ge 0}
  \end{equation}
  is continuous (equivalently, locally constant). 
\end{theorem}
\begin{proof}
  The continuity of \Cref{eq:agxcont} is equivalent to that of
  \begin{equation}\label{eq:cgxcont}
    X/\bG\ni \bG x
    \xmapsto{\quad}
    \dim C(\bG x)_{\alpha}
    = \dim C(\bG/\bG_x)_{\alpha}
    \xlongequal[\quad]{\text{\cite[pp.82-83]{rob}}}
    \dim\alpha\cdot \dim\alpha^{\bG_x},
  \end{equation}
  and we use switch between the two freely.
  
  \begin{enumerate}[label={},wide=0pt]
  \item {\bf ($\xRightarrow{\quad}$)} $C(X)_{\alpha}$ is a {\it Hilbert module} \cite[\S 1, p.4]{lnc_hilb} over $C(X/\bG)$ with the inner product
    \begin{equation*}
      \braket{f\mid g}:=\pi_{\cat{triv}}(f^*g)\in C(X/\bG)
      ,\quad
      f,g\in C(X).
    \end{equation*}
    Its assumed finite generation then entails its projectivity \cite[Corollary 15.4.8]{wo}, whence the conclusion: for any commutative ring $A$ and projective finitely-generated $A$-module $M$ the {\it rank} function \cite[\S 1]{vasc_proj-fg}
    \begin{equation*}
      \left\{\text{prime ideals of $A$}\right\} =: \Spec A
      \ni \fp
      \xmapsto{\quad}
      \dim_{A_{\fp}/\fp A_{\fp}}M_{\fp}
      \in \bZ_{>0}
    \end{equation*}
    (with $\bullet_{\fp}$ denoting {\it localization} \cite[\S 3]{am_comm}) is locally constant \cite[Proposition 1.4]{vasc_proj-fg} even in the (coarser, in our case of $A:=C(X/\bG)$ and $M:=C(X)_{\alpha}$) {\it Zariski} topology \cite[p.12]{am_comm}. 
    
  \item {\bf ($\xLeftarrow{\quad}$)} Conversely, the discussion following \cite[Theorem 2.2]{gog_top-fg} explains how $C(X)_{\alpha}$ can be recovered as the space of continuous {\it sections} ({\it cross}-sections in \cite[Definitions II.13.1]{fd_bdl-1}) of a {\it Hilbert bundle} \cite[Definition II.13.4]{fd_bdl-1} $F=F_{\alpha}$ over $X/\bG$. The fiber $F_{\bG x}$ over $\bG x\in X/\bG$ is easily seen to be precisely $C(\bG x)_{\alpha}$:
    \begin{itemize}[wide]
    \item the functions on $X$ annihilated, in the $C(X/\bG)$-module action on $C(X)$, by the annihilator of $C(\bG x)$ are tautologically precisely the elements of $C(\bG x)$;
    \item to which remark one can simply apply the functor $(-)_{\alpha}$.
    \end{itemize}
    The local constancy of \Cref{eq:cgxcont} thus entails \cite[Proposition 2.3]{dupre_hilbund-1} its {\it local triviality}. Yet another application of \cite[Theorem 1.6.3]{ros_algk} then finishes the proof (that the space $C(X)_{\alpha}$ of sections of $F$ is projective finitely generated over $C(X/\bG)$).
  \end{enumerate}  
\end{proof}

In particular, this recovers the aforementioned \cite[Remark 2, p.725]{dcy}:

\begin{corollary}\label{cor:iffviet}
  Given an action of a compact group $\bG$ on a compact Hausdorff space $X$, the representations $\alpha\in\Irr(\bG)$ are all tame precisely when the map $x\mapsto \bG_x$ of \Cref{eq:x2gx} is continuous for the Vietoris topology on closed subgroups of $\bG$.
\end{corollary}
\begin{proof}
  The backward implication ($\Leftarrow$) follows from \Cref{th:tameiffcontrk} and the continuity of
  \begin{equation*}
    \begin{tikzpicture}[auto,>=stealth,baseline=(current  bounding  box.center)]
      \path[anchor=base] 
      (0,0) node (l) {$\bH$}
      +(2,.5) node (u) {$\alpha^{\bH}$}
      +(4,0) node (r) {$\dim\alpha^{\bH}$,}
      ;
      \draw[|->] (l) to[bend left=6] node[pos=.5,auto] {$\scriptstyle $} (u);
      \draw[|->] (u) to[bend left=6] node[pos=.5,auto] {$\scriptstyle $} (r);
      \draw[|->] (l) to[bend right=6] node[pos=.5,auto,swap] {$\scriptstyle $} (r);
    \end{tikzpicture}
  \end{equation*}
  with $\bH$ ranging over the Vietoris-topologized space of closed subgroups of $\bG$ and $\alpha^{\bH}$ regarded as a subspace of the carrier space of $V$, and hence an element of the Grassmannian $\mathrm{Gr}(V)$.
  
  For the converse ($\Rightarrow$), suppose \Cref{eq:x2gx} is {\it not} continuous. Because {\it upper} semicontinuity is automatic \cite[Proposition 10.6]{hm5}, discontinuity at $x$ (and the compactness \cite[Theorem 1 and remark (IV) following it]{fell_haus-top} of the Vietoris topology) then means there is some convergent {\it net} \cite[Definition 11.2]{wil_top}
  \begin{equation*}
    x_{\lambda}
    \xrightarrow[\quad\lambda\quad]{}
    x
    \quad\text{with}\quad
    \bG_{x_{\lambda}}
    \xrightarrow[\lambda]{\quad\text{Vietoris}\quad}
    \bL\lneq \bG_x.
  \end{equation*}
  There is \cite[p.306]{reid} a non-trivial irreducible representation $\beta\in \Irr(\bG_x)$ with non-zero $\beta^{\bL}$, and also
  \begin{equation*}
    \alpha\in\Irr(\bG)
    \quad\text{with}\quad
    \beta\le \text{restriction }\big\downarrow^{\bG}_{\bG_x}\alpha
  \end{equation*}
  (by {\it Frobenius reciprocity} \cite[Theorem 8.9]{rob}: {\it induce} \cite[\S 8, p.84]{rob} $\beta$ up to $\bG$ and pick an appropriate irreducible summand therein). But then
  \begin{equation*}
    \dim\alpha^{\bG_{x_{\lambda}}}
    \xrightarrow[\quad\lambda\quad]{}
    \dim\alpha^\bL
    >
    \dim\alpha^{\bG_x},
  \end{equation*}
  contradicting tameness. \Cref{eq:x2gx} will thus indeed be continuous if all irreducible representations are tame. 
\end{proof}

\begin{remark}\label{re:exs.cont.isotr}
  Free actions, naturally, fall into the class delineated by \Cref{cor:iffviet}, as all $\bG_x$ are then trivial. So too do homogeneous actions $\bG\circlearrowright \bG/\bH$ for closed $\bH\le \bG$, the isotropy of $gH$ being the conjugate $g\bH g^{-1}$. 
\end{remark}

\cite[Proposition 2.8.2]{wat_indexcast} can be rephrased as saying that for finite $\bG$ acting on compact $X$ with
\begin{equation}\label{eq:x1def}
  X_{(1)}:=\left\{x\in X\ |\ \bG_x=\{1\}\right\}
  \subseteq X
\end{equation}
dense and proper, $\alpha\in\Irr(\bG)$ cannot {\it all} be tame. The following consequence of \Cref{th:tameiffcontrk} generalizes and sharpens that result. 

\begin{corollary}\label{cor:watgen}
  Let $X$ be a compact Hausdorff $\bG$-space for compact $\bG$, $X_{(1)}$ defined as in \Cref{eq:x1def}, and $x\in\overline{X_{(1)}}$.

  If $\alpha\in\Irr(\bG)$ restricts non-trivially to $\bG_x$, then it is not tame.  \qedhere
\end{corollary}

\section{Tube existence as a Lie-marking property}\label{se:tubext}

Given the noted \cite[Proposition 10.6]{hm5} upper semicontinuity of \Cref{eq:x2gx}, the triviality of $\bG_{x_0}$ means that the $\bG_x$ of nearby $x$ are contained in arbitrarily small neighborhoods of $1\in \bG$, hence \cite[Corollary 2.40]{hm5} trivial if $\bG$ is Lie (more generally, if $\bG$ is Lie then $\bG_x$ is conjugate to a subgroup of $\bG_{x_0}$ for $x\sim x_0$ \cite[Corollary II.5.5]{bred_cpct-transf}). In short, for compact $\bG$ acting on compact $X$ 
\begin{equation}\label{eq:glie}
  \bG\text{ is Lie }
  \xRightarrow{\quad}
  \left\{x\in X\ |\ \bG_x=\{1\}\right\}\subseteq X
  \text{ is open}.
\end{equation}
Nothing like this holds without the Lie constraint. In fact, for non-Lie $\bG$ the negation of \Cref{eq:glie} holds in the strongest possible form, per \Cref{th:nonlie-nonop}.

\begin{theorem}\label{th:nonlie-nonop}
  For a compact group $\bG$ the following conditions are equivalent.

  \begin{enumerate}[(a)]
  \item\label{item:th:nonlie-nonop:islie} $\bG$ is Lie.

  \item\label{item:th:nonlie-nonop:tube} Every orbit in every $T_{3\frac 12}$ (or just compact Hausdorff) $\bG$-space admits a tube. 
    
  \item\label{item:th:nonlie-nonop:h} For any compact Hausdorff $X$ acted upon by $\bG$ and any closed subgroup $\bH\le \bG$, the set
    \begin{equation*}
      X_{(\le \bH)}:=\left\{x\in X\ |\ \bG_x\text{ conjugate to a subgroup of }\bH\right\}\subseteq X
    \end{equation*}
    is open.

  \item\label{item:th:nonlie-nonop:1} For any compact Hausdorff $X$ acted upon by $\bG$ the set
    \begin{equation*}
      X_{(1)}
      :=
      \left\{x\in X\ |\ \bG_x=\{1\}\right\}\subseteq X
    \end{equation*}
    is open. 
  \end{enumerate}
\end{theorem}
\begin{proof}
  The downward implications are easily handled: \Cref{item:th:nonlie-nonop:islie} $\Rightarrow$ \Cref{item:th:nonlie-nonop:tube} is \cite[Theorem II.5.4]{bred_cpct-transf}, \Cref{item:th:nonlie-nonop:tube} $\Rightarrow$ \Cref{item:th:nonlie-nonop:h} follows as in the proof of \cite[Corollary II.5.5]{bred_cpct-transf}, and \Cref{item:th:nonlie-nonop:h} $\Rightarrow$ \Cref{item:th:nonlie-nonop:1} is obvious. The interesting implication is thus \Cref{item:th:nonlie-nonop:1} $\Rightarrow$ \Cref{item:th:nonlie-nonop:islie}.
  
  $\bG$ acts on each 
  \begin{equation*}
    \text{Lie quotient }\bG/\bK
    ,\quad
    \{1\}\ne \bK\trianglelefteq \bG,
  \end{equation*}
  hence also on their disjoint union and on
  \begin{equation*}
    X:=\text{{\it universal $\bG$-equivariant compactification} \cite[\S 2.8]{dvr_puc_1977} of }\coprod_{\bK}\bG/\bK=:Y.
  \end{equation*}
  Observe that the elements of $Y$ all have non-trivial isotropy groups (since we are ranging over {\it non-trivial} $\bK\trianglelefteq \bG$) and, $X$ being the closure of $Y$, the complement $X\setminus Y$ has empty interior. It remains to argue that the complement contains elements with trivial isotropy provided $\bG$ is not Lie.
  
  Recall next \cite[Corollary 2.43]{hm5} that
  \begin{equation*}
    \bG\text{ not Lie}
    \xRightarrow{\quad}
    \bG\cong \varprojlim_{\bK}\bG/\bK.
  \end{equation*}
  Ordering the neighborhoods $U\ni 1\in \bG$ by reverse-inclusion, we can construct a net
  \begin{equation*}
    U\xmapsto{\quad}1_U:=1\in \bG/\bK_U\subset Y
    \quad\text{for some}\quad
    \{1\}\ne \bK_U\trianglelefteq\bG
    \quad\text{contained in }U.
  \end{equation*}
  The {\it cluster points} \cite[Definition 11.3]{wil_top} of $(1_U)$ in $X$ (which do exist \cite[Theorem 17.4]{wil_top} by compactness) all belong to $X\setminus Y$ under the present assumption that $\bG$ is not Lie: sufficiently small neighborhoods $U\ni 1$ will fail to contain any given $\{1\}\ne \bK\trianglerighteq \bG$ with Lie quotient $\bG/\bK$, so no subnet can converge to a point in the component $\bG/\bK$ of $Y$. I claim that every such cluster point $p\in X\setminus Y$ has trivial isotropy.
  
  To see this, fix $1\ne s\in \bG$. For a sufficiently small neighborhood $U_0\ni 1\in \bG$ the image $s\in \bG/\bK_{U_0}$ is non-trivial, so we can define a bounded continuous function
  \begin{equation}\label{eq:fony}
    Y=\coprod_{\bK}\bG/\bK
    \xrightarrow{\quad f\quad}
    [0,1]
  \end{equation}
  by restricting some
  \begin{equation*}
    \bG/\bK_{U_0}
    \xrightarrow{\quad f_0\quad}
    [0,1]
    ,\quad
    f(1)=0,\ f(s)=1
  \end{equation*}
  along the surjection $\bG/\bK\xrightarrowdbl{}\bG/\bK_{U_0}$ whenever $\bK\le \bK_{U_0}$, and extending by zero to the other $\bG/\bK$. The resulting function \Cref{eq:fony}
  \begin{itemize}
  \item satisfies the uniform continuity requirements of \cite[Lemma 2.2]{dvr_puc_1977} and hence \cite[Propositions 2.6, 2.7 and 2.9]{dvr_puc_1977} extends continuously to all of $X$;

  \item vanishes at $p$ because it vanishes along the net $(1_U)_{U\supseteq U_0}$ clustering at $p$;

  \item and takes the value $1$ at $s\triangleright p$ because it does along the net
    \begin{equation*}
      (s\triangleright 1_U)_{U\supseteq U_0}
      =
      (\text{image of }s\text{ in }\bG/\bK_U)_{U\supseteq U_0}.
    \end{equation*}
  \end{itemize}
  This means that indeed $s\triangleright p\ne p$ (for {\it arbitrary} $1\ne s\in \bG$), and we are done. 
\end{proof}

In roughly the same spirit of action misbehavior for non-Lie groups, recall (\cite[Theorem II.5.8]{bred_cpct-transf}, \cite[Theorem 10.34]{hm5}) that for compact Lie $\bG$ any action on locally compact Hausdorff $X$ with {\it constant orbit type} $\bH\le \bG$ (i.e. with all stabilizers conjugate to $\bH$) results in a {\it fiber bundle} 
\begin{equation*}
  X
  \xrightarrowdbl{\quad}
  X/\bG
  \quad\text{with fiber }\bG/\bH
\end{equation*}
in the sense of \cite[\S II.1]{bred_cpct-transf}, which source in particular assumes local triviality. In particular, a free action gives a {\it principal $\bG$-bundle} with total space $X$, the term once again subsuming local triviality (in concert also with \cite[\S 14.1]{td_alg-top}, say). This is again not the case in general, for $\bG$ non-Lie: \cite[\S III]{zbMATH03184517} gives an example with compact abelian $\bG$, and a large class of examples can be produced by means of either \Cref{ex:nonlienotpb} or \Cref{le:allprofinite}.

\begin{example}\label{ex:nonlienotpb}
  Take $\bG=\prod_i \bG_i$, a product of an arbitrary infinite family of compact groups $\bG_i\ne \{1\}$. Each $\bG_i$ operates freely on the {\it $n^{th}$ truncation} $X_i:=E_{n_i}\bG_i$ of its {\it classifying space} \cite[\S 7.2.7]{hjjm_bdle} $E\bG$, so the product $\bG$ acts in the obvious fashion, freely again, on the  product $X:=\prod_i X_i$ (compact, because the $X_i$ are).
  
  Suppose the set $\{n_i\}_i$ of truncation parameters is unbounded. By \cite[\S 7.2.5]{hjjm_bdle} a local trivialization of the action would yield a $\bG$-equivariant map $X\to E_n\bG$ for some $n$, i.e. the action would have {\it finite $\bG$-index} \cite[Definition 6.2.3]{mat_bu} (as argued in \cite[Theorem 3.4]{cf1} for $\bG=\bZ/2$ for instance). This, then, would give $\bG_i$-equivariant maps
  \begin{equation*}
    E_{n_i}\bG_i
    \cong
    X_i
    \cong
    X_i\times (p_j)_{j\ne i}
    \lhook\joinrel\xrightarrow{\quad}
    X
    \xrightarrow{\quad}
    E_n\bG
    \xrightarrow{\quad}
    E_n\bG_i,
  \end{equation*}
  for fixed points $p_j\in X_j$, contradicting the Borsuk-Ulam-type \cite[Theorem 1.1]{cdt_be-hs} for $n_i>n$ (or \cite[Theorem 6.2.5]{mat_bu} for {\it finite} $\bG_i$). 
\end{example}

\begin{definition}\label{def:canrepemb}
  For a compact group $\bG$, the {\it canonical embedding} is the morphism
  \begin{equation*}
    \bG
    \lhook\joinrel\xrightarrow{\quad \iota=\iota_{\bG}:=(\alpha)_{\alpha} \quad}
    \bU
    :=
    \prod_{\alpha\in\Irr(\bG)}U(\dim\alpha)
    ,\quad
    U(n):=\text{$n\times n$ unitary group}
  \end{equation*}
  obtained by collating all individual $\alpha\in\Irr(\bG)$ regarded as maps $\bG\xrightarrow{\alpha}U(\dim\alpha)$.
  
  It is indeed an {\it embedding}, since all (locally) compact groups have ``enough'' irreducible representations \cite[Theorem 22.12]{hr-1}. 
\end{definition}

\begin{lemma}\label{le:allprofinite}
  A profinite group $\bG$ is finite if and only if its canonical embedding $\bG\lhook\joinrel\xrightarrow{\iota_{\bG}}\bU$ makes $\bU$ the total space of a principal $\bG$-bundle.
\end{lemma}
\begin{proof}
  The direction ($\Rightarrow$) is of course \cite[Theorem 10.34]{hm5} valid for arbitrary Lie groups, while the converse ($\Leftarrow$) is what the argument in the paragraph following \cite[Theorem 1]{most_sect} (ostensibly concerned with infinite profinite subgroups of infinite-dimensional tori $\left(\bS^1\right)^{S}$ for sets $S$) in fact proves: the product $\bU$ in \Cref{def:canrepemb} is connected and locally connected, so it cannot have open subsets homeomorphic to
  \begin{equation*}
    V\times \bG
    ,\quad
    V\subseteq \bU/\bG\text{ open and }
    \bG\text{ infinite profinite}
  \end{equation*}
  (because such sets in turn have no connected open subsets). 
\end{proof}

The construction in the proof of \Cref{th:nonlie-nonop} provides a natural free $\bG$-space to work with; it will turn out to be trivial (i.e. of the form $\bG\times Y$ with left-hand-factor action), but this is perhaps not immediately obvious in the non-Lie case. In any event, we take a detour to examine the matter.

For compact $\bG$ we write 
\begin{itemize}
\item $\beta_{\bG}X$ for the above-mentioned \cite[\S 2.8]{dvr_puc_1977} (universal) $\bG$-equivariant compactification of a Tychonoff $\bG$-space (so that plain $\beta X:=\beta_{\{1\}}X$ is the usual {\it Stone-\v{C}ech compactification} of $X$ \cite[Definition 19.4]{wil_top});

\item and $\gamma_{\bG}X := \beta_{\bG}X\setminus X$ for what we term the $\bG$-equivariant {\it corona} of the $\bG$-space $X$ (following the analogous non-equivariant terminology in \cite[Introduction, p.125]{zbMATH03857961}, say).
\end{itemize}

\begin{definition}\label{def:orbspace}
  Let $\bG$ be a compact group.
  
  \begin{enumerate}[(1)]
  \item A closed normal $\bK\trianglelefteq \bG$ is {\it co-Lie} if the quotient $\bG/\bK$ is Lie. 
    
  \item The {\it Lie $\bG$-space} (or {\it Lie (action) space of $\bG$}) is
    \begin{equation*}
      \Lambda(\bG):=\coprod_{\text{co-Lie }\bK\trianglelefteq \bG} \bG/\bK,
    \end{equation*}
    equipped with the obvious $\bG$-action on each $\bG/\bH$. 

  \item The {\it compactified} Lie $\bG$-space is the universal $\bG$-equivariant compactification $\beta_{\bG}\Lambda(\bG)$.
  \end{enumerate}  
\end{definition}

\begin{proposition}\label{pr:freeliegspace}
  Let $\bG$ be a compact group. The subspace
  \begin{equation*}
    \beta_{\bG}\Lambda(\bG)_{(1)}
    :=
    \left\{x\in \beta_{\bG}\Lambda(\bG)\ |\ \bG_x=\{1\}\right\}
    \subseteq
    \beta_{\bG}\Lambda(\bG)
  \end{equation*}
  is closed and trivial as a $\bG$-space. Specifically, writing
  \begin{equation*}
    \left(\cCL(\bG),\ \preceq\right)
    :=
    \left(\text{co-Lie  closed normal }\bK\trianglelefteq \bG,\quad \text{reverse inclusion}\right),
  \end{equation*}
  we have a $\bG$-equivariant isomorphism $\beta_{\bG}\Lambda(\bG)_{(1)}\cong \bG\times Y$ for
  \begin{equation}\label{eq:yincorona}
    Y
    :=
    \left\{\text{cluster points of the net }(\bK)_{\bK}\subset \cCL(\bG)\right\}
    \subseteq \beta\cCL(\bG).
  \end{equation}
\end{proposition}
\begin{proof}
  There are two qualitatively distinct cases to consider.

  \begin{enumerate}[(I),wide=0pt]
  \item {\bf $\bG$ is Lie.} The poset $(\cCL(\bG),\preceq)$ (ordered, recall, by {\it reverse} inclusion $\preceq\ :=\ \supseteq$) has a largest element $\{1\}\le \bG$. The only cluster point of the net $(\bK)_{\bK\in\cCL(\bG)}$ in the sense of \cite[Definition 11.3]{wil_top} is then $\{1\}$ itself, and the statement is interpreted as the claim that
    \begin{equation*}
      \beta_{\bG}\Lambda(\bG)_{(1)}\cong \bG
      \quad
      \text{$\bG$-equivariantly}.
    \end{equation*}
    More precisely, I claim that in this case $\beta_{\bG}\Lambda(\bG)_{(1)}$ is exactly the component $\bG\cong \bG/\{1\}$ of $\Lambda(\bG)$. Since of course no points in any other $\bG/\bK\subseteq \Lambda(\bG)$, $\bK\ne \{1\}$ have trivial isotropy, an equivalent formulation is that (for Lie $\bG$)
    \begin{equation*}
      \beta_{\bG}\Lambda(\bG)_{(1)}
      \cap\left(\text{$\bG$-equivariant corona }\gamma_{\bG}\Lambda(\bG)\right)
      =\emptyset.
    \end{equation*}
    To verify this, consider a convergent net
    \begin{equation*}
      \Lambda(\bG)\supseteq 
      \bG/\bK_{\alpha}\ni
      p_{\alpha}
      \xrightarrow[\alpha\nearrow]{\qquad}
      p\in \gamma_{\bG}\Lambda(\bG)
      \quad\text{for}\quad
      \bK_{\alpha}\in \cCL(\bG).
    \end{equation*}
    The right-hand ambient space in the following display being compact \cite[Theorem 1 and remark (IV) following it]{fell_haus-top}, the net
    \begin{equation*}
      (\bK_{\alpha})_{\alpha}
      \subset
      \left\{\text{closed normal }\bK\le \bG\right\}
      \text{ with the Vietoris topology}
    \end{equation*}
    has a cluster point $\bK\trianglelefteq \bG$. There is an origin neighborhood $U\ni 1\in \bG$ containing no non-trivial subgroups ($\bG$ has {\it no small subgroups} \cite[Corollary 2.40]{hm5}, being Lie). Because the $p_{\alpha}$ converge to a point {\it outside} $\Lambda(\bG)$,
    \begin{equation*}
      \bK_{\alpha}\cap (\bG\setminus U)\ne\emptyset
      \text{ for large }\alpha
      \ \xRightarrow{\quad}\ 
      \bK\cap (\bG\setminus U)\ne\emptyset
      \ \xRightarrow{\quad}\ 
      \bK\ne \{1\}.
    \end{equation*}
    But then $\{1\}\ne \bK\le \bG_p$, and we are done. 
    
  \item {\bf $\bG$ is not Lie.} This time around the co-Lie $\bK\trianglelefteq \bG$ are all non-trivial, the points in $\bG/\bK\subset \Lambda(\bG)$ all have $\bK$-isotropy, so trivial-isotropy points, if any, {\it must} belong to the corona $\gamma_{\bG}\Lambda(\bG)$. 
    
    First identify the space $Y$ of \Cref{eq:yincorona} with the subspace of $\gamma_{\bG}\Lambda(\bG)$ consisting of cluster points of the net $(1_{\bG/\bK})_{\bK\in\cCL(\bG)}$. We have already argued in the proof of \Cref{th:nonlie-nonop} that $Y\subseteq \Lambda(\bG)_{(1)}$. We now have an embedding
    \begin{equation}\label{eq:gyinl1}
      \bG\times Y
      \ni
      (s,y)
      \xmapsto[\quad\cong\quad]{}
      sy
      \in
      \left(\coprod_{s\in \bG}\text{cluster points of }(s\bK)_{\bK}\right)
      \subseteq
      \beta_{\bG}\Lambda(\bG)_{(1)},
    \end{equation}
    and it will suffice to argue that that embedding is onto. Consider, then, a convergent net
    \begin{equation*}
      \bG/\bK_{\lambda}
      \ni
      s_{\lambda}\bK_{\lambda}
      \xrightarrow[\quad\lambda\quad]{}
      x\in \beta_{\bG}\Lambda(\bG)_{(1)}.
    \end{equation*}
    Because $x$ has trivial isotropy, the net $(\bK_{\lambda})_{\lambda}$ must (Vietoris-)converge to $\{1\}$. Next, convergence in $\beta_{\bG}\Lambda(\bG)$ also entails (as in the proof of \Cref{th:nonlie-nonop}, via \cite[Lemma 2.2]{dvr_puc_1977}) the convergence of $(s_{\lambda})$ to some $s\in \bG$. But then $x$ belongs to the image of $\{s\}\times Y$ through \Cref{eq:gyinl1}, completing the proof. 
  \end{enumerate}
\end{proof}

\addcontentsline{toc}{section}{References}



\Addresses

\end{document}